\documentclass[12pt]{article}

\def\@eqnnum{\rm ([section].\theequation)}
\usepackage{amsthm}
\usepackage{amsfonts}
\usepackage{pictexwd,dcpic}
\usepackage{amsmath}
\usepackage{amssymb}
\usepackage{color}
\usepackage{amsbsy}
\numberwithin{equation}{section}

\newtheorem{df}{Definition}[section]
\newtheorem{p}[df]{Proposition}
\newtheorem{corol}[df]{Corollary}
\theoremstyle{definition}

\newtheorem{ex}[df]{Example}

\title{\bf  Polynomials in algebraic analysis}
\author{Piotr Multarzy\'{n}ski\\{\it Faculty of Mathematics and Information Science}\\
 {Warsaw University of Technology} \\
{\it 00-661 Warsaw, Pl. Politechniki 1, Poland}\\
e-mail: multarz@mini.pw.edu.pl}
\date{}
\begin{document}
\maketitle
\begin{abstract}
The concept of polynomials in the sense of algebraic analysis, for a single right invertible linear operator, was introduced and studied originally by D. Przeworska-Rolewicz \cite{DPR}. One of the elegant results corresponding with that notion is a purely algebraic version of the Taylor formula, being a generalization of its usual counterpart, well known for functions of one variable. 
In quantum calculus there are some specific discrete derivations analyzed, which are right invertible linear operators \cite{kac}. Hence, with such quantum derivations one can associate the corresponding concept of algebraic polynomials and consequently the quantum calculus version of Taylor formula \cite{MULT2}. In the present paper we define and analyze, in the sense of algebraic analysis, polynomials corresponding with a given family of right invertible operators. Within this approach we generalize the usual polynomials of several variables.
\end{abstract}
\section{Introduction}
Algebraic analysis, founded and developed by D. Przeworska-Rolewicz \cite{DPR}, 
is an algebra-based theory unifying many different generalizations of derivatives and integrals, not necessarily continuous. The main concepts of this algebraic formulation are right invertible linear operators, their right inverses and associated initial operators. 
Right invertible operators are considered to be algebraic  counterparts of derivatives and their right inverses together with  initial operators correspond with the idea of integration. Amongst many examples one can interpret in terms of algebraic analysis there are e.g. usual differential and integral calculus, generalized differential calculus in rings \cite{Moh} and many different quantum calculi \cite{kac, MULT1,  Hart, Jackson}.  With a fixed right invertible operator $D$, defined in a linear space $X$, one can naturally associate the concept of  $D$-polynomials and algebraically formulate Taylor formula \cite{DPR, MULT2}. However, the $D$-polynomials, Taylor formula, definite integrals associated with a single right invertible operator $D$ constitute the algebraic counterparts corresponding with mathematical analysis for functions of one variable. Therefore, there is a natural need to extend algebraic analysis in order to algebraically generalize ideas from mathematical analysis for functions of many variables. To begin this direction, 
we  replace a single operator $D$ by a fixed family $\cal D$ of right invertible operators and study the corresponding $\cal D$-polynomials.
\section{Preliminaries}
Let $X$ be a linear space over a field $\mathbb K$ and ${\cal L}(X)$
be the family of all linear mappings $D:U\rightarrow V$, for any
$U$, $V$ - linear subspaces of $X$. We shall use the notation
$dom(D)=U$, $codom(D)=V$ and $im\, D=\{Du: u\in U\}$ for the domain,
codomain and image of $D$, correspondingly.

Throughout this paper we use the notation
\begin{equation}
{\mathbb N}=\{1,2,3,\ldots\}\;\;\; \mbox{and}\;\;\; {\mathbb
N}_0=\{0,1,2,3,\ldots\}\; .
\end{equation}

Whenever $D_1=\ldots =D_m=D\in {\cal L}(X)$, we shall write
$D^m=D_1\ldots D_m$, for $m\in\mathbb N$,  and additionally
$D^0=I\equiv id_{dom(D)}$.
By the space of constants for $D\in{\cal L}(X)$  we shall mean the family
\begin{equation}\label{ZD}
 Z(D)=ker\,D\, .
\end{equation}

For any $D\in{\cal L}(X)$ and  $m\in\mathbb N$, we assume the
notation
\begin{equation}
Z_0(D)=ker D\setminus\{0\}\;\;\;\mbox{and}\;\;\;Z_m(D)=ker D^{m+1}\setminus ker
D^{m}\; .
\end{equation}
 Evidently, for any $D\in{\cal L}(X)$
there is
\begin{equation}\label{Zdisjoint}
Z_i(D)\cap Z_j(D)=\emptyset ,
\end{equation}
whenever $i\neq j$.
\begin{p}\label{linindep}
Let $D\in{\cal L}(X)$, $m\in{\mathbb N}_0$, and $Z_i(D)\neq\emptyset$
for $i=0,\ldots ,m$. Then, any elements $u_i\in Z_i(D)$, $i=0,\ldots
, m$, are linearly independent.
\end{p}
\begin{proof}
Consider a linear combination $u=\lambda_0 u_0+\ldots +\lambda_m u_m$
and suppose that $u=0$ for some coefficients $\lambda_0,\ldots
,\lambda_m\in{\mathbb K}$. Hence we obtain the sequence of
equations: $D^{k}u=\lambda_{k}D^{k}u_{k}+\ldots +\lambda_m
D^{k}u_{m}=0$, for $k=0,\ldots ,m$. Step by step, from these
equations we compute $\lambda_m=0,\ldots ,\lambda_0 =0$.
\end{proof}

Let us define
\begin{equation}\label{calR(X)}
{\cal R}(X)=\{D\in{\cal L}(X): codom(D)=im\,D\},
\end{equation}
 i.e. each element
$D\in{\cal R}(X)$ is considered to be a surjective mapping (onto its
codomain). Thus, ${\cal R}(X)$ consists of all right invertible
elements.
\begin{df}\label{Rinverse}
 An operator $R\in{\cal L}(X)$ is said to be a right inverse of
 $D\in{\cal R}(X)$ if $dom(R)=im(D)$ and $DR=I\equiv id_{im(D)}$.
By ${\cal R}_{D}$ we denote the family of all right inverses of $D$.
\end{df}
In fact, ${\cal R}_{D}$ is a nonempty family, since for each $y\in
im(D)$ we can select an element $x\in D^{-1}(\{y\})$ and define
$R\in{\cal R}_{D}$ such that $R: y\mapsto x$.

The fundamental role in the calculus of right invertible operators
play the so-called initial operators, projecting the domains of
linear operators onto the corresponding space of their constants.
\begin{df}\label{initial}
Any operator $F\in{\cal L}(X)$, such that $dom(F)=dom(D)$,
$im(F)=Z(D)$ and $F^2=F$ is said to be an initial operator induced
by $D\in{\cal R}(X)$. We say that an initial operator $F$
corresponds to a right inverse $R\in{\cal R}_{D}$ whenever $FR=0$ or
equivalently if
\begin{equation}\label{FbyR}
F=I-RD\, .
\end{equation}
The family of all initial operators induced by $D$ will be denoted
by ${\cal F}_D$.
\end{df}
The families ${\cal R}_{D}$ and ${\cal F}_D$ uniquely determine each
other. Indeed, formula (\ref{FbyR}) characterizes initial operators
by means of right inverses, whereas formula
\begin{equation}\label{RbyF}
R=R'-FR'\, ,
\end{equation}
which is independent of $R'$, characterizes right inverses by means
of initial operators. Both families ${\cal R}_D$ and ${\cal F}_D$
are fully characterized by formulae
\begin{equation}\label{characterizeR}
{\cal R}_D=\{R+FA: dom\,A=im\,D,\, A\in{\cal L}(X)\},
\end{equation}
\begin{equation}\label{characterizeF}
{\cal F}_D=\{F(I-AD): dom\,A=im\,D,\, A\in{\cal L}(X)\},
\end{equation}
where    $R\in{\cal R}_{D}$ and $F\in{\cal F}_{D}$ are fixed
arbitrarily.
\vspace{3mm}

Let us illustrate the above concepts with two basic examples.
\begin{ex}\label{derivative} $X={\mathbb R}^{\mathbb R}$ - the linear space of all
functions, $D\in{\cal R}(X)$ - usual derivative, i.e.
 $Dx(t)\equiv x'(t)$, with $dom(D)\subset X$ consisting of all differentiable functions.
 Then, for an arbitrarily fixed $a\in\mathbb R$, by formula
 $Rx(t)=\int\limits_{a}^{t}x(s)ds$ one can define a right inverse $R\in{\cal
R}_{D}$ and the initial operator $F\in{\cal F}_D$ corresponding to
$R$ is given by $Fx(t)=x(a)$.
\end{ex}
\begin{ex}\label{difference}  $X={\mathbb {\mathbb
R}^{\mathbb N}}$ - the linear space of all sequences, $D\in{\cal
R}(X)$ - difference operator, i.e.
 $(Dx)_n=x_{n+1}-x_n$, for $n\in\mathbb N$. A right inverse $R\in{\cal
 R}_{D}$ is defined by the formulae $(Rx)_{1}=0$ and $(Rx)_{n+1}=\sum\limits_{i=1}^{n}x_{i}$  while
 $(Fx)_n=x_1$ defines the initial operator $F\in{\cal F}_D$ corresponding to $R$.
\end{ex}
An immediate consequence of Definition \ref{initial}, for an
invertible operator $D\in{\cal R}(X)$, i.e. $ker\, D=\{0\}$, is that
${\cal F}_{D}=\{0\}$. Therefore, the nontrivial initial operators do
exist only for operators which are right invertible but not
invertible. The family of all such operators is then
\begin{equation}\label{nontriv}
{\cal R}^+(X)=\{D\in{\cal R}(X): dim\, Z(D) >0 \}.
\end{equation}
\begin{p}[Taylor Formula]
Suppose $D\in{\cal R}(X)$ and let $F\in{\cal F}_{D}$ be an initial
operator corresponding to $R\in{\cal R}_{D}$. Then the operator
identity
\begin{equation}\label{taylor}
I=\sum_{k=0}^{m}R^kFD^k + R^{m+1}D^{m+1}\, ,
\end{equation}
holds on $dom(D^{m+1})$,  for $m\in{\mathbb N}_0$.
\end{p}
\begin{proof}(Induction) See Ref.\cite{DPR}.
\end{proof}
Equivalent identity,   expressed as
\begin{equation}\label{xtaylor}
x=\sum_{k=0}^{m}R^kFD^kx + R^{m+1}D^{m+1}x\, ,
\end{equation}
for $x\in dom(D^{m+1})$ and $m\in{\mathbb N}_0$, is an algebraic
counterpart of the Taylor expansion formula,  commonly known in
mathematical analysis, for functions of one variable. The first component of formula (\ref{xtaylor})
reflects the polynomial part while the second one can be viewed as
the corresponding remainder.

\begin{ex}To clearly demonstrate the resemblance of formula (\ref{xtaylor})
with the commonly used Taylor expression, we take $D$, $R$ and $F$
as in Example \ref{derivative}. Since there are many forms of the
remainders in use, it is more interesting to calculate the polynomial
part, which gives the well known result
$\sum\limits_{k=0}^{m}R^kFD^kx(t)=\sum\limits_{k=0}^{m}\frac{x^{(k)}(a)}{k!}(t-a)^{k}$,
for any function $x\in dom(D^{m+1})$.
\end{ex}
\begin{p}\label{nillpot}
Let $D\in {\cal R}(X)$ and $R\in{\cal R}_{D}$. Then $R$ is not a
nilpotent operator.
\end{p}
\begin{proof}
Suppose that $R^n\neq 0$ and $R^{n+1}=0$, for some $n\in\mathbb N$.
Then $0\neq R^n=IR^n=DRR^n=DR^{n+1}=0$, a contradiction.
\end{proof}
\begin{p}\label{nonemptyZm}
If $D\in{\cal R}^+(X)$, then $Z_m(D)\neq\emptyset$, for any
$m\in\mathbb N_0$.
\end{p}
\begin{proof}
The relation $Z_0(D)\neq\emptyset$ is straightforward. Let
$R\in{\cal R}_{D}$ and $z\in Z_0(D)$ be arbitrarily chosen elements.
Then, for any $m\in\mathbb N$, there is $R^{m}z\in Z_m(D)$.
\end{proof}
In algebraic analysis, with any right invertible operator $D\in{\cal R}^+(X)$  we
associate the following concept of $D$-polynomials.
\begin{df}\label{deg-polynomials}
If $D\in{\cal R}^+(X)$, then any element $u\in Z_{m}(D)$ is said
to be a $D$-polynomial of degree m, i.e. $deg\, u = m$, for
$m\in{\mathbb N}_0$. We assign degree $deg\, u=-\infty$ to the zero polynomial
$u=0$.
\end{df}
\begin{p}\label{polsumRz}
If $D\in{\cal R}^+(X)$ and $R\in{\cal R}_{D}$, then for any
$D$-polynomial $u\in Z_{m}(D)$ there exist elements $z_0,\ldots
,z_m\in Z_0(D)$ such that
\begin{equation}\label{uRz}
u=z_0+Rz_1+\ldots +R^mz_m .
\end{equation}
\end{p}
\begin{proof}By formula (\ref{xtaylor}) we can write the
identity $u=\sum_{k=0}^{m}R^kFD^ku$ since $u\in Z_{m}(D)$ and
$R^{m+1}D^{m+1}u=0$. Then we define elements $z_k=FD^ku$,
$k=0,\ldots ,m\,$, which ends the proof.
\end{proof}

\begin{df}\label{monomials} Let $D\in{\cal R}^+(X)$ and $R\in{\cal R}_{D}$.
Then, any element $R^kz\in Z_{k}(D)$, for $z\in Z_0$, is said to
be an $R$-homogeneous $D$-polynomial (or $R$-monomial) of degree
$k\in{\mathbb N}_0$.
\end{df}
 Thus, any $D$-polynomial $u\in Z_{m}(D)$, of degree $deg\, u =
m$, is a sum of linearly independent $R$-homogeneous elements
$R^kz_k$, $k=0,\ldots ,m$. The linear space of all $D$-polynomials
is then
\begin{equation}\label{PolD}
P(D)=\bigcup\limits_{k=0}^{\infty}ker\, D^k
\end{equation}
whereas
\begin{equation}\label{PolDn}
P_n(D)=\bigcup\limits_{k=0}^{n+1}ker\, D^k=ker\, D^{n+1}
\end{equation}
 is the linear space of all $D$-polynomials of degree at most
$n\in{\mathbb N}_0$.\newline
Let us fix a basis
\begin{equation}\label{zetabasis}
\{\zeta_s\}_{s\in S}\subset Z(D)\; ,
\end{equation}
  of the linear space $Z(D)$, $D\in{\cal R}^+(X)$, and define
\begin{equation}
Z^s(D)=Lin\{\zeta_s\}\, ,
\end{equation}
for $s\in S$. Then
\begin{equation}\label{ZDdirect}
Z(D)=\bigoplus_{s\in S}Z^s(D)\, .
\end{equation}
\begin{p}\label{polbasis}
For an arbitrary right inverse $R\in{\cal R}_{D}$, the family
$\{R^m\zeta_s\colon {s\in S}, m\in{\mathbb N}_0\}$ is the basis of
the linear space $P(D)$. Naturally,  $\{R^m\zeta_s\colon {s\in S},
m=0,1,\ldots ,n\}$ forms the basis of the linear space $P_n(D)$, for
$n\in{\mathbb N}_0$.
\end{p}
\begin{proof}Let $u=\sum_{i=1}^{k}\sum_{s\in S_{i}}a_{is} R^{m_i}\zeta_{s}$,
 $m_1<\ldots <m_k$ and  $S_i\subset S$ be finite subsets for
$i=1,\ldots ,k$. Assume $u=0$ and calculate $D^{m_k}u=\sum_{s\in
S_{k}}a_{ks}\zeta_{s}=0$, which implies $a_{ks}=0$, for all $s\in
S_k$. Hence $u=\sum_{i=1}^{k-1}\sum_{s\in S_{i}}a_{is}
R^{m_i}\zeta_{s}=0$ and analogously we get $D^{m_{k-1}}u=\sum_{s\in
S_{k-1}}a_{(k-1)s}\zeta_{s}=0$, which implies $a_{(k-1)s}=0$, for
all $s\in S_{k-1}$. Similarly we prove that $a_{is}=0$, for all
$s\in S_{i}$, $i=k-2,\ldots ,1$. Now, let $u\in P(D)$ be a
polynomial of degree $deg\,u=n \in{\mathbb N}_0\}$. Then, on the
strength of Proposition \ref{polsumRz}, we can write
$u=\sum_{k=0}^{n}R^kz_k $, for some elements $z_0,\ldots ,z_n\in
Z(D)$. In turn, each element $z_k$ can be expressed as a linear
combination $z_k=\sum_{s\in S_k}a_{ks}\zeta_s$, for some finite
subset of indices $S_k\subset S$. Hence we obtain
$u=\sum_{k=0}^{n}\sum\limits_{s\in S_k}a_{ks}R^k\zeta_s\,$.
\end{proof}

With a right inverse $R\in{\cal R}_{D}$, $s\in S$ and $n\in{\mathbb
N}_0$, we shall associate the linearly independent family
$\{R^m\zeta_s\colon m\in\{0,\ldots ,n\}\}$ forming a basis of the
linear space of $s$-homogeneous $D$-polynomials
\begin{equation}\label{Vsndirect}
V_{s}^{n}(D)=Lin\{R^m\zeta_s\colon m\in\{0,\ldots ,n\}\}\, ,
\end{equation}
(independent of the choice of $R$) of dimension
\begin{equation}\label{dimVsn}
dim\,V_{s}^{n}(D)=n+1\, ,
\end{equation}
being a linear subspace of $P_n(D)$. Then, on the strength of
Proposition \ref{polbasis}, the linear space $P_n(D)$ is a direct
sum
\begin{equation}\label{Pndirect}
P_n(D)=\bigoplus_{s\in S}V_s^n(D)\, .
\end{equation}
\begin{corol}If $dim\, Z(D)<\infty$, the following formula holds
\begin{equation}\label{dimpoln}
dim\, P_n(D)=(n+1)\cdot dim\, Z(D)\, ,
\end{equation}
for any $n\in{\mathbb N}_0$.
\end{corol}
Naturally, one can extend formula (\ref{Vsndirect}) and define
\begin{equation}\label{Vdirect}
V_{s}(D)=Lin\{R^m\zeta_s\colon m\in{\mathbb N}_0\}\, ,
\end{equation}
which is both $D$- and $R$-invariant subspace of $P(D)$, i.e.
\begin{equation}\label{DinvariantVsn}
DV_{s}(D)\equiv\{Du\colon u\in V_{s}(R)\}\subset V_{s}(D)\, ,
\end{equation}
\begin{equation}\label{RinvariantVsn}
RV_{s}(D)\equiv\{Ru\colon u\in V_{s}(D)\}\subset V_{s}(D)\, .
\end{equation}
Thus,  $P(D)$ turns out to be simultaneously $D$- and $R$-invariant
linear subspace of $X$, since it can be decomposed as the following
direct sum
\begin{equation}\label{PDVsdirect}
P(D)=\bigoplus_{s\in S}V_s(D)\, .
\end{equation}
 Since $P(D)$ is a linear subspace of $X$, there exists
(not uniquely) another linear subspace $Q(D)$ of $X$ such that
\begin{equation}\label{PsumQ}
X=P(D)\oplus Q(D)\, .
\end{equation}
Then, every linear mapping $\phi\colon X\rightarrow X$ can be
decomposed as the direct sum
\begin{equation}\label{phiPQ}
\phi=\phi_{_P}\oplus\phi_{_Q}\, ,
\end{equation}
of two restrictions $\phi_{_P}=\phi_{|P(D)}$ and
$\phi_{_Q}=\phi_{|Q(D)}$, i.e. for any $x'\in P(D)$ and $x''\in
Q(D)$ there is $\phi (x'+x'')=\phi_{_P}(x')+\phi_{_Q}(x'')$. In
particular, the mappings  $D\in{\cal R}^+(X)$, $R\in{\cal R}_{D}$
can be decomposed as direct sums $D=D_{_P}\oplus D_{_Q}$,
$R=R_{_P}\oplus R_{_Q}$ such that
\begin{equation}
I=DR=D_{_P}R_{_P}\oplus D_{_Q}R_{_Q}=I_{_P}\oplus I_{_Q}\, ,
\end{equation}
\begin{equation}
RD=R_{_P}D_{_P}\oplus R_{_Q}D_{_Q}\, ,
\end{equation}
which allows for the decomposition of the initial operator $F$
corresponding to $R$
$$
F=I-RD=I_{_P}\oplus I_{_Q}-R_{_P}D_{_P}\oplus R_{_Q}D_{_Q}=
$$
\begin{equation}
=(I_{_P}-R_{_P}D_{_P})\oplus (I_{_Q}-R_{_Q}D_{_Q})=F_{_P}\oplus
F_{_Q}.
\end{equation}
\begin{p}
Let $D\in{\cal R}^+(X)$, $R',R''\in{\cal R}_{D}$ be any right
inverses and $F',F''\in{\cal F}_{D}$ be the initial operators
corresponding to $R'$ and $R''$, respectively. Then
$R:=R'_{_P}\oplus R''_{_Q}\in{\cal R}_{D}$ and $F:=F'_{_P}\oplus
F''_{_Q}\in{\cal F}_{D}$ corresponds to $R$.
\end{p}
\begin{proof}
$$DR=(D_{_P}\oplus D_{_Q})(R'_{_P}\oplus
R''_{_Q})=D_{_P}R'_{_P}\oplus D_{_Q}R''_{_Q}=I_{_P}\oplus I_{_Q}=I\,
,$$
$$RD=(R'_{_P}\oplus R''_{_Q})(D_{_P}\oplus D_{_Q})=R'_{_P}D_{_P}\oplus R''_{_Q}D_{_Q}\,
,$$

$$F=F'_{_P}\oplus F''_{_Q}=(I_{_P}-R'_{_P}D_{_P})\oplus
(I_{_Q}-R''_{_Q}D_{_Q})=
$$
$$=
I_{_P}\oplus I_{_Q}-R'_{_P}D_{_P}\oplus R''_{_Q}D_{_Q}=I-RD\, .
$$
\end{proof}
The last results allow one to combine right inverses and initial
operators as direct sums of independent components.
\section{$\cal D$-polynomials}
Originally, the main concepts of algebraic analysis have been defined  for a single right invertible operator $D\in{\cal R}^{+}(X)$. Such an approach one can recognize as an algebraic generalization of mathematical analysis dedicated to functions of one variable.
In this section we propose to replace a single $D\in{\cal R}^{+}(X)$ by a nonempty family ${\cal D}\subset{\cal R}^{+}(X)$ and extend the notion of a polynomial in order to relate it with mathematical analysis of many variables. Since mathematical analysis of many variables forms a groundwork for differential geometry, there is a hope that the corresponding geometric ideas can also be studied in terms of algebraic analysis. 

By ${\cal R}_D$ and ${\cal F}_D$ we shall denote the families of all right inverses and all initial operators defined for a single $D\in{\cal R}^{+}(X)$, correspondingly. This we understand as $R: X\rightarrow dom\, D$, $DR=I$, and $F: X\rightarrow ker\, D$, $F^2=F$, $im\, F=ker\, D$,  for any $R\in{\cal R}_D$ and $F\in{\cal F}_D$. 
If the additional condition $FR=0$ is fulfilled, we say that $F$ corresponds to $R$. 
\begin{df}\label{Dconstants}
For any family ${\cal D}\subset{\cal L}(X)$ we define
\begin{equation}\label{ZcalD}
Z({\cal D})=\bigcap\limits_{ D \in {\cal D}}ker D,
\end{equation}
which is called the family of $\cal D$-constants.
In the trivial case, i.e. when ${\cal D}=\emptyset$, we assume 
\begin{equation}\label{ZforemptyD}
Z({\emptyset})=X\, .
\end{equation}
\end{df}
Obviously, for any two families ${\cal D'},{\cal D''} \subset{\cal L}(X)$, 
\begin{equation}\label{ZD'capD''}
Z({\cal D'}\cup {\cal D''})=Z({\cal D'})\cap Z({\cal D''}).
\end{equation}
On the strength of formula (\ref{ZD'capD''}) we notice that 
\begin{equation}\label{ZDinclusion}
{\cal D'}\subset{\cal D''}\Rightarrow Z({\cal D'})\supset Z({\cal D''}).
\end{equation}
\begin{df}
A family $\cal D$ separates elements of $X$ if
for any two different elements $x_1,x_2\in X$ there exists $D\in\cal D$ such that $Dx_1\neq Dx_2$.
\end{df}
In particular, $\cal D$ separates elements of $X$ if there is an invertible element $D\in{\cal D}$. 
If a family $\cal D$ separates elements of $X$, then $Z({\cal D})$ becomes trivial, i.e. $Z({\cal D})=\{0\}$. 

In the sequel we shall assume ${\cal D}\subset {\cal R}(X)$ to be a  family of right invertible operators which does not separate elements of $X$, i.e. $Z({\cal D})\neq\{0\}$. In particular, from this assumption we notice that 
${\cal D}\neq\emptyset$ and any element $D\in\cal D$ is right invertible but not invertible, i.e. ${\cal D}\subset {\cal R}^+(X)$.
\begin{df}\label{initial}
An initial operator for  $\cal D$ is any map $F:X\rightarrow Z({\cal D})$, satisfying $F^2=F$ and $im\, F=Z({\cal D})$. By ${\cal F}({\cal D})$ we shall denote the family of all initial operators for $\cal D$.
\end{df}
To compare two notations, notice that  ${\cal F}(\{D\})={\cal F}_D$, when ${\cal D}=\{D\}$.\vspace{3mm}

The immediate consequence of the above definition is that nontrivial initial operators for $\cal D$ do exist only if 
\begin{equation}\label{nontrivial}
dim\, Z({\cal D}) >0.
\end{equation}
 Therefore, in the sequel we shall assume condition (\ref{nontrivial}) to be fulfiled.\\
\begin{df}
Let ${\cal D}\subset{\cal R}^{+}(X)$,  ${\cal D}\neq\emptyset$. We say that an initial operator $F\in{\cal F}({\cal D})$ corresponds with a family  $\{ R_D \}_{D\in{\cal D}}$ of right inverses  $R_D\in{\cal R}_D$, if $FR_D=0$, for any $D\in\cal D$.
\end{df}
\begin{p}
Let $D\in\cal D$ and $R\in{\cal R}_D$. Then, for any initial operator $F\in{\cal F}({\cal D})$, the operator $R-FR$ is also a right inverse of $D\,$.
\end{p}
\begin{proof}
Indeed, $D(R-FR)=DR-DFR=DR=I$.
\end{proof}
In the particular case when ${\cal D}=\{D\}$, $D\in{\cal R}^{+}(X)$ and $F\in{\cal F}({\cal D})$,  the right inverse $R-FR\in{\cal R}_D$ is independent of $R\in{\cal R}_D$ (compare Ref. \cite{DPR}). However, if ${\cal D}\subset{\cal R}^{+}(X)$ consists of more than one element, the operator $R-FR\in{\cal R}_D$ depends on $R$, in general.
\begin{p} For any family $\{ R_D \}_{D\in{\cal D}}$, $R_D\in{\cal R}_D$,  an initial operator $F\in{\cal F}({\cal D})$ corresponds with the family $\{R_D-FR_D\}_{D\in\cal D}$. 
\end{p}
\begin{proof}
Indeed, $F(R_D-FR_D)=FR_D-F^2R_D=FR_D-FR_D=0$, for any $D\in{\cal D}$.
\end{proof}
\begin{df}\label{PnD}
The linear subspace of $X$ defined as
\begin{equation}\label{PncalD}
P_n({\cal D})=\bigcap\limits_{ (D_0,\ldots ,D_{n}) \in {\cal D}^{n+1}}ker(D_0\ldots D_{n}),
\end{equation}
for any $n\in{\mathbb N}_0$,  is called the space of $\cal D$-polynomials of degree at most $n$.  For the sake of convenience we assume 
\begin{equation}\label{Pnempty}
P_n(\emptyset)=X,
\end{equation}
for any $n\in{\mathbb N}_0$, in case when ${\cal D}=\emptyset$,.
\end{df}
One can easily notice that 
\begin{equation}\label{Pninclusion}
 Z({\cal D})=P_0({\cal D})\subset P_1({\cal D})\subset P_2({\cal D})\subset\ldots.
\end{equation}
 Thus, we can build the linear subspace of $X$
\begin{equation}
P({\cal D})=\bigcup\limits_{n\in{{\mathbb N}_0}}P_n({\cal D}),
\end{equation}
called  the linear space of all $\cal D$-polynomials.

Let us notice that formula (\ref{ZD'capD''}) can be generalized as follows
\begin{equation}\label{PnD'capD''}
P_n({\cal D'}\cup{\cal D''})\subset P_n({\cal D'})\cap P_n({\cal D''}),
\end{equation}
for any $n\in{\mathbb N}_0$.
Consequently, for any two nonempty families ${\cal D'},{\cal D''} \subset{\cal R^{+}}(X)$, we get  the generalization of formula (\ref{ZDinclusion}), namely  
\begin{equation}\label{PnDinclusion}
{\cal D'}\subset{\cal D''}\Rightarrow P_n({\cal D'})\supset P_n({\cal D''}).
\end{equation}
 If a family $\cal D$ separates elements of $X$, we obtain $ Z({\cal D})=P_0({\cal D})= P_1({\cal D})=P_2({\cal D})=\ldots =\{0\}$. Therefore, we shall be interested with families $\cal D$ which do not separate elements of $X$.
\begin{df}\label{ZnDdefinition}
Define the families $Z_0({\cal D})=P_0({\cal D})\setminus \{0\}$ and
\begin{equation}
Z_n({\cal D})=P_n({\cal D})\setminus P_{n-1}({\cal D}),
\end{equation}
for any $n\in\mathbb N$. Then, any element $x\in Z_n({\cal D})$ is said to be a $\cal D$-polynomial of degree $deg(x)=n\in{\mathbb N}_0$. By definition, we shall assume  $deg(0)=-\infty$.
\end{df}
Evidently, in case when $\cal D$ separates elements of $X$, there are no nontrivial $\cal D$-polynomials, i.e. $ Z_n({\cal D})=\emptyset$, for any $n\in{\mathbb N}_0$. On the strength of formula (\ref{Pnempty}) and Definition ({\ref{ZnDdefinition}) we can notice that 
\begin{equation}
Z_n(\emptyset )=\left\{
\begin{array}{ccl}
   X\setminus\{0\}& \mbox{if} &n=0\\
 \emptyset& \mbox{if} &n>0\,  , 
\end{array}\right.
\end{equation}
i.e. all nonzero elements of $X$ are considered to be of degree $0$, if ${\cal D}=\emptyset$. 
\begin {df}\label{effectiveD}
Let ${\cal D}\subset{\cal R}^+(X)$ be a family which does not separate elements of $X$. Then,  ${\cal D}$ is said to be effective if there extists a family, so-called $\cal D$-proper family, $\{R_D\}_{D\in{\cal D}}$  of right inverses $R_D\in{\cal R}_D$, $D\in\cal D$, such that for any ${\cal D'}\subset{\cal D}$, $n\in{\mathbb N}_0$,
\begin{equation}\label{D-eff-formula}
R_{D}(Z_n({\cal D}\setminus{\cal D'}))\subset\left\{
\begin{array}{ccl}
    Z_{n+1}({\cal D}\setminus{\cal D'})&\mbox{if} &D\in{\cal D}\setminus{\cal D'}\\
    Z_{n}({\cal D}\setminus{\cal D'})& \mbox{if} &D\in{\cal D'}\; .
\end{array}\right.
\end{equation}

By definition we assume that the empty family $\cal D=\emptyset$ is effective.
\end{df}
\begin{p}
Any subfamily ${\cal D}_1\subset\cal D$ of an effective family $\cal D$ is effective as well.
\end{p}
\begin{proof}By definition, the empty subfamily $\emptyset ={\cal D}_1\subset\cal D$ is effective.
 Now, let 
$\emptyset\neq{\cal D}_1\subset\cal D$ and ${\cal D'}_1\subset{\cal D}_1$ be an arbitrary subfamily. Since an effective family $\cal D$ does not separate elements of $X$, the more so does not its subfamily ${\cal D}_1$. With subfamilies ${\cal D'}_1, {\cal D}_1$ we can associate the family ${\cal D'}=({\cal D}\setminus{\cal D}_1)\cup{\cal D'}_1$.  Then, one can easily verify that ${\cal D}\setminus{\cal D'}={\cal D}_1\setminus{\cal D'}_1$ and by formula (\ref{D-eff-formula}), for any $D\in{\cal D}_1$, we obtain
\begin{equation}\label{D_1-eff-formula}
R_{D}(Z_n({\cal D}_1\setminus{\cal D'}_1))\subset\left\{
\begin{array}{ccl}
    Z_{n+1}({\cal D}_1\setminus{\cal D'}_1)&\mbox{if} &D\in{\cal D}_1\setminus{\cal D'}_1\\
    Z_{n}({\cal D}_1\setminus{\cal D'}_1)& \mbox{if} &D\in{\cal D'}_1\; .
\end{array}\right.
\end{equation}
\end{proof}
\begin{p}
If ${\cal D}\subset{\cal R}^+(X)$ is a nonempty effective family which does not separate points of $X$, then $Z_n({\cal D})\neq\emptyset$, for any $n\in{\mathbb N}_0$.
\end{p}
\begin{proof}Since $\cal D$ does not separate points of $X$, we get $Z_0({\cal D})\neq\emptyset$.  Now, on the strength of Definition (\ref{effectiveD}), we can take  an arbitrary $D\in\cal D$,   ${\cal D'}=\emptyset$, and by formula (\ref{D-eff-formula}) we obtain
$\emptyset\neq R_D^{n}Z_0({\cal D})\subset R_D^{n-1}Z_1({\cal D})\subset\ldots\subset R_D^{0}Z_{n}({\cal D})= Z_{n}({\cal D})$. Hence  $Z_{n}({\cal D})\neq\emptyset$, for any $n\in{\mathbb N}_0$. 
\end{proof}
Below we formulate the concept of a weakly effective family $\cal D$, which is also sufficient to produce nonempty sets of $\cal D$-polynomials of degree $n$.
\begin{df} Let ${\cal D}\subset{\cal R}^+(X)$ and $D\in\cal D$. 
We say that $D$ is a regular element of $\cal D$ (or $\cal D$-regular,  for short) if there exists a right inverse $R\in{\cal R}_D$ such that  
\begin{equation}\label{proper}
\forall_{D'\in{\cal D}\setminus\{D\}}\;
R(ker\, D')\subset ker\, D'\; .
\end{equation}
A right inverse $R\in{\cal R}_D$ is said to be $\cal D$-proper if  condition (\ref{proper}) is fulfilled.
\end{df}
\begin{df}
A  family ${\cal D}\subset{\cal R}^+(X)$ is said to be weakly effective if there exists a $\cal D$-regular element $D\in\cal D$. 
\end{df}

Evidently, for any $D\in{\cal R}^+(X)$, the family ${\cal D}=\{D\}$ is weakly effective. Moreover, in such a case any element $R\in{\cal R}_D$ is a $\cal D$-proper right inverse.
Indeed, in this case ${\cal D}\setminus\{D\}=\emptyset$ and formula (\ref{proper}) is fulfilled trivially (by logic). 
\begin{p}
Suppose ${\cal D}\subset{\cal R}^+(X)$ is a weakly effective family which does not separate elements of $X$. Then 
$Z_n({\cal D})\neq\emptyset$, for any $n\in{\mathbb N}_0$.
\end{p}
\begin{proof}
Since family ${\cal D}$ does not separate elements of $X$, we get $dim Z({\cal D})>0$ 
and consequently $Z_0({\cal D})=Z({\cal D})\setminus\{0\}\neq\emptyset$. 
Now, let $D\in{\cal D}$ be a regular element of $\cal D$ and $R\in{\cal R}_D$ be its ${\cal D}$-proper right inverse. This means that, for any $z\in Z_0({\cal D})$ and $n\in\mathbb N$, there is $R^nz\in ker D'$, for any $D'\in{\cal D}\setminus\{D\}$.
On the other hand, $D^nR^nz=z$ and $D^{n+1}R^nz=0$, which means that 
$R^nz\in Z_n({\cal D})$, for $z\neq 0$ and $n\in\mathbb N$.
\end{proof}
The following statement shows the hierarchy relation between the two effectivity concepts, which motivates the term "weakly".
\begin{p}
Any nonempty effective family $\cal D$ is weakly effective.
\end{p}
\begin{proof}
Let $\cal D$ be a nonempty effective family. A singleton ${\cal D}=\{D\}$ is evidently an efective family. Let us assume now that 
$D, D'\in\cal D$, $D\neq D'$, and ${\cal D'}={\cal D}\setminus\{D'\}$. 
Since $D\in\cal D'$, by formula (\ref{D-eff-formula}), 
for $n=0$, we obtain 
$R_{D}(Z_0({\cal D}\setminus{\cal D'}))\subset Z_0({\cal D}\setminus{\cal D'})$. The last equality can be easily extended to
$R_{D}(Z({\cal D}\setminus{\cal D'}))\subset Z({\cal D}\setminus{\cal D'})$.
 Finally, because ${\cal D}\setminus{\cal D'}=\{D'\}$, we can write $R_{D}(ker D')\subset ker D'$. \end{proof}

Below we present two examples showing some effective families ${\cal D}\subset{\cal R}^+(X)$ and their $\cal D$-polynomials. 
In both examples use the notation $x\cdot y=x_1y_1+...+x_my_m$ and $x_{|k}=(x_1,...,x_{k-1},x_{k+1},...,x_m)\in\mathbb R^{m-1}$, for any
$x=(x_1,...,x_m), y=(y_1,...,y_m)\in\mathbb R^m$.
 \begin{ex} Let $X$ be the linear space of  functions $f\colon{\mathbb R^m}\rightarrow\mathbb R$, for a fixed  $m\in\mathbb N$, and let ${\cal D}=\{D_1,\ldots ,D_m\}$ be the family of the generalized derivations \cite{Moh}, i.e. $D_k=a_k+\frac{\partial}{\partial x_k}$, $a_k\in\mathbb R$. 
The elements of $Z(\{D_k\})$ are of the form $f(x)=C(x_{|k})e^{-a_kx_k}$ and consequently 
$Z({\cal D})$ consists of functions $f(x)=Ce^{-a\cdot x}$,
where $a, x\in{\mathbb R}^m$ and $C\in\mathbb R$. Hence, $dim Z({\cal D})>0$ and the family $Z({\cal D})$ does not separate elements of $X$. The following formula defines a $\cal D$-proper family of right inverses $R_k\in{\cal R}_{D_k}$, 
\begin{equation}\label{R_k}
R_kf(x)=e^{-a_kx_k}\cdot\int\limits_{0}^{x_k}f(x_1,\ldots ,x_{k-1},t,x_{k+1},\ldots ,x_m)\cdot e^{a_kt}dt\, , 
\end{equation}
for $k=1,\ldots ,m$. Hence, the family $\cal D$ is effective. Since $Z_0({\cal D})=Z({\cal D})\setminus\{0\}$, the polynomials of degree zero are simply functions $f(x)=Ce^{-a\cdot x}$, where $a, x\in{\mathbb R}^m$ and $C\in{\mathbb R}\setminus\{0\}$. 
In turn, if we apply formula (\ref{R_k}) to $f(x)=e^{-a\cdot x}$, we can obtain $\cal D$-polynomials of higher degrees. For example $R_kf(x)=x_ke^{-a\cdot x}$  are $\cal D$-polynomials of degree 1, whereas $R_k^2f(x)=\frac{1}{2}x_k^2e^{-a\cdot x}$ or $R_iR_jf(x)=x_ix_je^{-a\cdot x}$, $i\neq j$, are $\cal D$-polynomials of degree 2.
\end{ex}
\begin{ex} Let $X$ be a linear space of functions $x\colon{\mathbb N^m}\rightarrow\mathbb R$, for a fixed  $m\in\mathbb N$,
${\cal D}=\{D_1,\ldots ,D_m\}$, where $(D_kx)(i_1,...,i_m)=x(i_1,...,i_{k-1},i_k+1,i_{k+1},...,i_m)-x(i_1,...,i_m)$ and let us take the right inverses $R_k\in{\cal R}_{D_k}$ defined by $(R_kx)(i_1,...i_{k-1},1,i_{k+1},...,i_m)=0$, $(R_kx)(i_1,...i_{k-1},n+1,i_{k+1},...,i_m)=\sum\limits_{j=1}^{n}x(i_1,...i_{k-1},j,i_{k+1},...,i_m)$, for $n\geq 1$. One can check that  $\{R_k\}_{k=1}^{m}$ is a $\cal D$-proper family, which means that the family $\cal D$ of partial difference operators is effective.
Evidently, the space $Z({\cal D})$ consists of all constant functions, hence the elements of $Z_0({\cal D})$ are nonzero constant functions. In turn, functions $x(i_1,...,i_m)=x_j$, $j=1,...,m$, are the examples of $\cal D$-polynomials of degree 1, whereas functions $x(i_1,...,i_m)=i_k^2$ or $x(i_1,...,i_m)=i_ki_j$ are $\cal D$-polynomials of degree 2.\vspace{3mm}\newline
{\bf Remark.} Effectivity property plays a significant role in the algebraic formulation of the Taylor formula specified for a family $\cal D$ of many right invertible operators, which is the algebraic counterpart of the well known Taylor formula for functions of many variables.
\end{ex}

\end{document}